\documentclass[12pt]{amsart}

\usepackage{amssymb}
\usepackage{amsmath}
\usepackage{amsthm}
\usepackage{amscd}
\usepackage{mathrsfs}
\usepackage{graphicx}
\usepackage{fullpage}
\graphicspath{{Figures/}} 

\usepackage{cancel} 
\usepackage{enumerate} 
\usepackage{hyperref} 
\usepackage[usenames,dvipsnames]{color} 
\usepackage{polynom}
\usepackage{subfig} 

\newtheorem{thm}{Theorem}[section]

\newtheorem{cor}[thm]{Corollary}

\newtheorem{lem}[thm]{Lemma}
\newtheorem{prop}[thm]{Proposition}

\theoremstyle{definition}
\newtheorem{defn}[thm]{Definition}
\newtheorem{rem}[thm]{Remark}

\newtheorem{claim}[thm]{Claim}
\numberwithin{equation}{section}
\newcommand{\norm}[1]{\left\Vert#1\right\Vert}
\newcommand{\abs}[1]{\left\vert#1\right\vert}
\newcommand{\set}[1]{\left\{#1\right\}}
\newcommand{\Real}{\mathbb R}
\newcommand{\eps}{\varepsilon}
\newcommand{\To}{\longrightarrow}
\newcommand{\BX}{\mathbf{B}(X)}
\newcommand{\EE}{\mathcal{E}}
\newcommand{\LL}{\mathcal{L}}
\newcommand{\OO}{\mathcal{O}}
\newcommand{\AX}{\mathsf{A}(X)}
\makeatletter
\newcommand\xleftrightarrow[2][]{\ext@arrow 0099{\longleftrightarrowfill@}{#1}{#2}}
\def\longleftrightarrowfill@{\arrowfill@\leftarrow\relbar\rightarrow}
\makeatother

\def\mc{\mathcal}
\def\frk{\mathfrak}
\def\wt{\widetilde}
\def\wh{\widehat}
\def\mb{\mathbf}
\def\bb{\mathbb}
\def\ol{\overline}

\def\onto{\twoheadrightarrow}
\def\into{\hookrightarrow}
\def\iso{\stackrel{\simeq}{\longrightarrow}}
\def\from{\leftarrow}

\def\tensor{\otimes}
\def\empt{\varnothing}
\def\bt{\bullet}
\def\he{\simeq}
\def\del{\partial}
\def\delbar{\overline{\partial}}
\def\tr{\text{tr}}

\DeclareMathOperator{\Hom}{Hom}
\DeclareMathOperator{\Ext}{Ext}
\DeclareMathOperator{\End}{End}
\DeclareMathOperator{\Aut}{Aut}
\DeclareMathOperator{\Der}{Der}
\DeclareMathOperator{\Tor}{Tor}

\def\sHom{\mc{H}om}
\def\sExt{\mc{E}xt}
\def\sEnd{\mc{E}nd}

\DeclareMathOperator{\grHom}{grHom}
\DeclareMathOperator{\grExt}{grExt}
\DeclareMathOperator{\grEnd}{grEnd}

\DeclareMathOperator{\Sch}{\mathbf{Sch}}
\DeclareMathOperator{\Vect}{{\sf Vect}}
\DeclareMathOperator{\lmod}{-{\sf mod}}
\DeclareMathOperator{\rmod}{{\sf mod}-}
\DeclareMathOperator{\grmod}{-{\sf grmod}}

\DeclareMathOperator{\Perf}{\frk{Perf}}
\DeclareMathOperator{\Dcoh}{\mb{D}^b coh}
\DeclareMathOperator{\DSg}{DSg}
\DeclareMathOperator{\coh}{coh}
\DeclareMathOperator{\Qcoh}{Qcoh}

\DeclareMathOperator{\spec}{Spec}
\DeclareMathOperator{\maxspec}{MaxSpec}
\DeclareMathOperator{\proj}{Proj}
\DeclareMathOperator{\Spec}{\bb{S}pec}
\DeclareMathOperator{\Sym}{Sym}
\def\invlim{\underleftarrow{\lim}}
\DeclareMathOperator{\Pic}{Pic}
\DeclareMathOperator{\tot}{tot}
\DeclareMathOperator{\Ind}{Ind}

\DeclareMathOperator{\gr}{gr}
\DeclareMathOperator{\id}{id}
\DeclareMathOperator{\supp}{supp}
\DeclareMathOperator{\ann}{ann}
\DeclareMathOperator{\nil}{nil}
\DeclareMathOperator{\sing}{sing}
\DeclareMathOperator{\coker}{coker}
\DeclareMathOperator{\im}{im}
\DeclareMathOperator{\eval}{eval}
\DeclareMathOperator{\cn}{cone}
\DeclareMathOperator{\Rees}{Rees}
\DeclareMathOperator{\Bl}{Bl}
\DeclareMathOperator{\Rep}{Rep}
\DeclareMathOperator{\rad}{rad}

\def\vv{{\vee\vee}}
\def\rank{\text{rank}}
\def\dim{\text{dim}}
\DeclareMathOperator{\Gr}{Gr}

\def\codim{\text{codim}}
\DeclareMathOperator{\sgn}{sgn}
\DeclareMathOperator{\mult}{mult}

\def\dolb{\Omega^{0,\bullet}}
\DeclareMathOperator{\HH}{{\sf HH}}
\def\svC{\check{\mc{C}}}

\def\Gm{\mathbb{G}_m}
\def\Cx{\C^\times}
\DeclareMathOperator{\GL}{GL}
\DeclareMathOperator{\SL}{SL}
\DeclareMathOperator{\Aff}{Aff}
\def\pr{\text{pr}}
\def\ad{\text{ad}}

\def\Qu{\noindent (\textbf{Q}) }

\def\A{\mathbb{A}}
\def\C{\mathbb{C}}
\def\N{\mathbb{N}}
\def\P{\mathbb{P}}
\def\bP{\P}
\def\Q{\mathbb{Q}}
\def\R{\mathbb{R}}
\def\bbS{\mathbb{S}}
\def\V{\mathbb{V}}
\def\Z{\mathbb{Z}}

\def\cA{\mc{A}}
\def\cB{\mc{B}}
\def\cC{\mc{C}}
\def\cD{\mc{D}}
\def\cE{\mc{E}}
\def\cF{\mc{F}}
\def\cG{\mc{G}}
\def\cH{\mc{H}}
\def\cI{\mc{I}}
\def\cJ{\mc{J}}
\def\cK{\mc{K}}
\def\cL{\mc{L}}
\def\cM{\mc{M}}
\def\O{\mc{O}}
\def\cO{\mc{O}}
\def\cP{\mc{P}}
\def\cQ{\mc{Q}}
\def\cS{\mc{S}}
\def\cT{\mc{T}}
\def\cU{\mc{U}}
\def\cV{\mc{V}}
\def\cW{\mc{W}}
\def\cX{\mc{X}}
\def\cY{\mc{Y}}

\def\bk{\mathbf{k}}

\def\bD{\mathbf{D}}
\def\bH{\mathbf{H}}
\def\bL{\mathbf{L}}
\def\bR{\mathbf{R}}

\def\Ft{\wt{F}}
\def\Kt{\wt{K}}
\def\Mt{\wt{M}}
\def\Nt{\wt{N}}
\def\Qt{\wt{Q}}
\def\Rt{\wt{R}}
\def\St{\wt{S}}
\def\Tt{\wt{T}}
\def\Ut{\wt{U}}
\def\Wt{\wt{W}}
\def\Xt{\wt{X}}
\def\Yt{\wt{Y}}
\def\Zt{\wt{Z}}

\def\cEt{\wt{\cE}}

\def\Ah{\wh{A}}
\def\Bh{\wh{B}}
\def\Gh{\wh{G}}
\def\Lh{\wh{L}}
\def\Mh{\wh{M}}
\def\Nh{\wh{N}}
\def\Qh{\wh{Q}}
\def\Rh{\wh{R}}
\def\Xh{\wh{X}}
\def\Yh{\wh{Y}}
\def\Zh{\wh{Z}}

\def\cBh{\wh{\cB}}
\def\cJh{\wh{\cJ}}

\def\Bbr{\ol{B}}
\def\Fbr{\ol{F}}
\def\Gbr{\ol{G}}
\def\Kbr{\ol{K}}
\def\Lbr{\ol{L}}
\def\Mbr{\ol{M}}
\def\Qbr{\ol{Q}}
\def\Rbr{\ol{R}}
\def\Sbr{\ol{S}}
\def\Wbr{\ol{W}}
\def\Xbr{\ol{X}}
\def\Ybr{\ol{Y}}
\def\Zbr{\ol{Z}}

\def\sfA{{\sf A}}
\def\sfC{{\sf C}}
\def\sfD{{\sf D}}
\def\sfH{{\sf H}}
\def\sfR{{\sf R}}
\def\sfT{{\sf T}}
\def\sfU{{\sf U}}

\def\Sfr{\frk{S}}
\def\Xfr{\frk{X}}
\def\Yfr{\frk{Y}}

\def\vC{\check{C}}

\def\rmH{\mathrm{H}}

\newcommand{\reg}[2]{\ensuremath{{\rm reg}(#2,#1)}}
\newcommand{\creg}[2]{\ensuremath{{\rm reg}_{\rm cont}(#2,#1)}}
\newcommand{\rego}[1]{\reg{\cO(1)}{#1}}
\newcommand{\crego}[1]{\creg{\cO(1)}{#1}}
\newcommand{\deq}{\ensuremath{\stackrel{\rm def}{=}}}
\newcommand{\st}[1]{\ensuremath{\left\{ #1 \right\}}}

\begin{document}
\title{Rational Curves on Moduli Spaces of Vector Bundles}
\author{Yusuf Mustopa}
\address{Yusuf Mustopa, Tufts University, Department of Mathematics, Bromfield-Pearson Hall, 503 Boston Avenue, Medford, MA 02155, USA}
\address{Max-Planck-Institut f\"{u}r Mathematik, Vivatsgasse 7, 53111, Bonn, Germany}
\address{Department of Mathematics, University of Massachusetts Boston,100 Morrissey Boulevard, Boston, MA 02125, USA}
\email{Yusuf.Mustopa@umb.edu}
\author{Montserrat Teixidor i Bigas}
\address{Montserrat Teixidor i Bigas, Tufts University, Department of Mathematics, Bromfield-Pearson Hall, 503 Boston Avenue, Medford, MA 02155, USA}
\email{Montserrat.TeixidoriBigas@tufts.edu}

\begin{abstract} We completely describe the components of expected dimension of the Hilbert Scheme of rational curves of fixed degree $k$ in the moduli space ${\rm SU}_{C}(r,L)$
of  semistable vector bundles of rank $r$ and determinant $L$ on a curve $C$.
 We show that for every  $k \geq 1$ there are ${\rm gcd}(r, \deg L)$ unobstructed components. 
 In addition, if $k$ is divisible by $r_1(r-r_1)(g-1)$ for $1\le r_1\le r-1$, there is an additional obstructed component of the expected dimension for each such $r_1$.
 We construct families of obstructed components and show that their generic point is not the generic vector bundle of given rank and determinant.  
 Finally, we also obtain an upper bound on the degree of rational connectedness of ${\rm SU}_{C}(r,L)$ which is linear in the dimension.
\end{abstract}
\maketitle

\section*{Introduction}

Let $C$ be a smooth projective curve of genus $g \geq 2$, let $r$ and $d$ be integers with $r > 0,$ and let $L$ be a line bundle of degree $d$ on $C.$  Throughout the paper we write $h := {\rm gcd}(r,d).$
 
 The moduli space $M :={\rm SU}_{C}(r,L)$ parametrizes semistable rank-$r$ vector bundles on $C$ with determinant $L$  up to $S-$equivalence.  It is well-known that $M$ is a normal and locally factorial projective variety of dimension $(r^{2}-1)(g-1)$ whose singularities are at worst rational and Gorenstein.   Moreover, there is an ample divisor $\Theta$ on $M$ such that ${\rm Pic}{M} \cong {\Z}\Theta$ and $K_{M} = -2h\Theta$; in particular $M$ is Fano of index $2h$. One can show that $M$ is  smooth precisely when either $h=1$ or $g=r=2$ and $d$ is even; in all other cases, the singular locus is the (nonempty) locus of equivalence classes of strictly semistable bundles.  We refer to \cite{DN} for further details on the geometry of $M$. 
Rational curves have long been a useful tool in the study of varieties in general, and of Fano varieties in particular (e.g. \cite{K,Hu1}).  
The main topic of this paper is the structure of the Hilbert scheme ${\rm Mor}_{k}(\P^{1},M)$ parametrizing rational curves $f : \P^1 \to M$ of degree $k \geq 1.$  

This subject has a long tradition. Narasimhan-Ramanan  \cite{NR2} and Newstead \cite{N}, who addressed the case of $g=r=2$ and $d$ odd, gave beautiful geometric descriptions of the space of lines in $M.$  In \cite{S}, Sun classified  curves of minimal degree and determined the minimal degree of a rational curve through a generic point of $M$.  Additional results in this direction were obtained in \cite{MS}, where the authors give constructions of rational curves of minimal degree in $M$ and in \cite{Li} where the particular case of genus 3, rank 2 and even degree is described.

The aforementioned results all deal with very particular cases of genus and rank of the vector bundle or specific degree of the rational curve.  
Up to now, the only case that has been addressed for arbitrary genus and degree of the rational curve is when $r=2$ and $d$ is odd;
 in this case, Castravet \cite{C} classified the irreducible components of the Hilbert scheme of rational curves.
  Here, we extend much of Castravet's work to the setting of arbitrary rank and degree. 
   Our first two main results completely classify the components of ${\rm Mor}_{k}(\P^{1},M)$ which have the expected dimension:

\begin{thm}\label{thmain}
 	 For all $k \geq 1$, there are precisely $h$ components of ${\rm Mor}_{k}(\P^{1},M)$ which are unobstructed (and therefore of the expected dimension).
	 They correspond to either families of lines in spaces of one-step extensions of vector bundles on $C$ or extensions of a skyscraper sheaf by a vector bundle.
 \end{thm}
 
 \begin{thm}\label{thext}
 There is an obstructed component of ${\rm Mor}_{k}(\P^{1},M)$ having the expected dimension if and only if $k$ is divisible by $r_1(r-r_1)(g-1)$ for some $r_1, 1\le r_1\le r-1$.  
 There is a unique such component for each $r_1$, and it corresponds to families of rational curves of higher degrees in spaces of extensions of vector bundles on $C$.  
 \end{thm}
 
 We also determine some obstructed components of ${\rm Mor}_{k}(\P^{1},M)$ that are not of the expected dimension, and we show that their associated rational curves only fill up a proper closed subvariety of $M$, 
 i.e.~ that the generic point on any of these rational curves is not a generic stable vector bundle of the given rank and determinant.   

\begin{thm}\label{thoth}
The  components of ${\rm Mor}_{k}(\P^{1},M)$ not listed in Theorems \ref{thmain} and \ref{thext} are obstructed components 
corresponding to rational curves of higher degree  in spaces of extensions of vector bundles on $C$ or to rational curves in multiple step extensions of vector bundles.
The points of these rational curves correspond to vector bundles on $C$ that fill a proper subvariety of $M$.
 \end{thm}

Rational curves of minimal degree and their tangent directions have been used to study the deformation theory of $M$ (e.g.~ \cite{Hu2,HuR}).
 There has been a lot of interest in studying uniruled varieties, that is varieties covered by rational curves. 
One of the most interesting invariants of a uniruled variety is the minimum degree of a rational curve through a generic point of the variety.
This invariant was determined for $M$ in \cite{S}; a natural next step is to ask for the minimal degree of a rational curve between two generic points \cite{KMM}.
 We answer this question here  looking at rational connectivity and computing the minimum degree of an irreducible rational curve containing two generic points of $M$ 
 (Proposition \ref{prop:rat-conn}).

An important question about rational curves on a Fano variety $X$ comes from Batyrev's conjecture on the growth rate of the number of components of ${\rm Mor}_{k}(\P^{1},X)$ 
as $k$ increases (e.g. \cite{LT}).  We hope that our description of components will add to the reservoir of examples against which to check this conjecture.

In Section \ref{secprel}, we review an equation giving the degree of a curve in $M$ using the natural decomposition of a vector bundle on ${\mathbb P}^1$ as direct sum of line bundles.
We also prove a criteria that allows to describe all families of rational curves in $M$ in terms of extensions.
In Section \ref{secuncomp}  we construct the unobstructed components of the Hilbert space and show that these are all the unobstructed components (see Theorem \ref{teornumcomp}).
Our main tool is the stability of a generic extension of generic vector bundles of given rank and degree.
In Section \ref{secadcomp}, we construct the additional components of Theorem \ref{thext}. We also construct additional families of vector bundles and prove Theorem \ref{thoth}.
 Finally, in Section  \ref{secratc2points}, we consider the degree of rational curves containing two generic points in $M$.

\medskip
\textbf{Acknowledgments:}  We would like to thank Ana-Maria Castravet, Brian Lehmann, Sukhendu Mehrotra, Swarnava Mukhopadyay, Peter Newstead, and Hacen Zelaci for valuable discussions and correspondence related to this work.   The first author was supported by the Max-Planck-Institut f\"{u}r Mathematik while part of this work was carried out; he would like to thank them for their hospitality and excellent working conditions. 
     
\section{Preliminaries}\label{secprel}

In what follows, we fix a line bundle $L$ of degree $d$ on $C$ and denote ${\rm SU}_{C}(r,L)$ by $M.$

The Zariski tangent space to the moduli space ${U}_{C}(r,d)$ parametrizing semistable bundles of rank $r$ and degree $d$ at a point corresponding to a stable bundle $E$ 
 may be naturally identified with $H^1(C, E^* \otimes E)$. 
 The trace map ${\rm tr} :  E^* \otimes E \to \cO_{C}.$ induces a decomposition 
\[ E^* \otimes E\cong {\mathcal O}_C\oplus {\cH}om_0(E)\]
where $\cO_{C}$ corresponds to homotheties and ${\cH}om_0(E)$ denotes traceless endomorphisms of $E$.  The derivative of the determinant map ${\rm det}_{r,d} : U_{C}(r,d) \to M$ can be identified with the map
 \begin{equation*}
 	H^{1}({\rm tr}) : H^{1}(C, E^* \otimes E) \to H^{1}(C,\cO_{C})
 \end{equation*}
induced by the trace map ${\rm tr}: E^* \otimes E \to \cO_{C}.$

  The tangent space to $M$ at $E$ can be identified with $H^1(C, {\cH}om_0(E))$, which has dimension $(r^{2}-1)(g-1).$
  Consider the determinant map  ${\rm det}_{r,d}: U_{C}(r,d) \to {\rm Pic}^{d}(C)$. 
  The fibers of ${\rm det}_{r,d}$ are all isomorphic to $M,$ as one can go from one fiber to any other by tensoring with a suitable line bundle of degree zero.  
 
Throughout the paper, we will write $h:=(r,d), \overline{r}:=r/h$ and $\overline{d} := d/h.$ 
 If $F$ is a semistable vector bundle of rank $\overline{r}$ and degree $\overline{r}(g-1)-\overline{d},$ then $E \otimes F$ is semistable of slope $g-1$. 
 For a generic choice of $F$, the locus $\{E \in M : h^{0}(E \otimes F) > 0\}$ is a {\bf proper} subset of $M$ and  then the support of an ample Cartier divisor. 
   The linear equivalence class of this divisor, which is denoted by $\Theta,$ is independent of $F.$  It is known that ${\rm Pic}(M) \cong \Z{\Theta}$ and that $K_{M} = -2h\Theta.$

\begin{defn} 
	\label{defdeg}  
	If $C'$ is a smooth projective curve of genus $g'$ and $f : C' \to M$ is a morphism, the degree of $f$ is $\deg(f) :=  c_{1}(f^*\Theta)$.
\end{defn}

We are interested in the Hilbert scheme ${\rm Mor}_{k}(C',M)$ parametrizing morphisms from $C'$ to $M$ of degree $k \geq 1$.  It is well-known that the Zariski tangent space to the Hilbert scheme at $f$ is $H^0(C',f^{\ast}T_{M})$, and that $f$ is unobstructed if $H^{1}(C',f^{\ast}T_{M})=0.$  

\begin{lem}
	The expected dimension of a component of ${\rm Mor}_{k}(C',M)$ is its  minimum possible dimension 
	\[ 2hk + (r^{2}-1)(g-1)(1-g')\]
\end{lem}

\begin{proof}
	We have from Riemann-Roch that
\begin{equation*}
	\chi(C',f^{\ast}T_{M}) = {\rm deg}(f^{\ast}T_{M}) + \rank (f^{\ast}T_{M})(1-g') =  {\rm deg}(f^{\ast}T_{M}) + {\rm dim}(M)(1-g') 
\end{equation*}
\begin{equation*}
	= -f_{\ast}[C'] \cdot K_{M} + (r^{2}-1)(g-1)(1-g') = 2hk + (r^{2}-1)(g-1)(1-g')
\end{equation*}
Thus $2hk + (r^{2}-1)(g-1)(1-g')$ is the minimal dimension of a component of the Hilbert scheme and in fact the expected dimension. 
  \end{proof}
 
 For the rest of the paper, we focus on the case in which $C'={\mathbb P}^1$.
 
There is a projective bundle $\mathcal P$ and a vector bundle $\cA_{0}$ on $C\times M$  such that for all $E \in M,$ we have ${\mathcal P}_{|C\times \{ E\}} \cong {\mathbb P}(E)$ and ${\cH}om_0(E) \cong {\mathcal A}_{0}|_{C\times \{ E\}}$. 
  If $p_{M} : C \times M \to M$ is the projection map, we also have $T_{M} \cong R^{1}p_{M\ast}\cA_{0}.$ 
 There exists a vector bundle $\widetilde{\cE}$ on $C \times M$ such that $\widetilde{\cE}|_{C \times [E]} \cong E$ for all $[E] \in M$ 
 (in particular, $\cP \cong \mathbb{P}(\widetilde{\cE})$) and $\cA_{0} \cong {\cH}om_{0}(\widetilde{\cE})$ precisely when $h=1$; in this case, $\widetilde{\cE}$ is a Poincar\'e sheaf on $C \times M.$  
 
 The following is proved in  \cite{S} Lemma 2.1. We include it here  for ease of citation.
 \begin{lem}
 	\label{sun-lem}
 	For any $f \in {\rm Mor}_{k}(\P^1,M)$ there exists a vector bundle ${\mathcal E}$ on $C\times {\mathbb P}^1$ such that 
	${\mathcal  E}_{|C\times \{ t\}}=f(t)$ for all $t \in \P^1$ and ${\cH}om_0({\mathcal E})=(1_{C} \times f)^*({\mathcal A}_0)$. \hfill \qedsymbol
 \end{lem}

This allows us to identify non-constant maps $f:{\mathbb P}^1\to M$ with vector bundles on $C\times {\mathbb P}^1$ 
that restrict to a semistable vector bundle of rank $r$ and determinant $L$ on every fiber of the projection to ${\mathbb P}^1$, even when $h > 1.$

  \begin{lem} 
  	\label{unobs} 
	The point of ${\rm Mor}_{k}(\P^{1},M)$ corresponding to a morphism $f : \P^{1} \to M$ whose image lies in the smooth locus of $M$ is unobstructed if and only if the restriction of the associated vector bundle $\cE$ on $C\times {\mathbb P}^1$
  	 to the generic fiber $\{ P \}\times {\mathbb P}^1$ is isomorphic to  
  ${\mathcal E}_{|\{ P \}\times {\mathbb P}^1}={\mathcal O}_{{\mathbb P}^1}(\alpha)^{r_1}\oplus {\mathcal O}_{{\mathbb P}^1}(\alpha -1)^{r-r_1}$ for some $\alpha \in \Z$ and $r_1 \in [1,r].$
 \end{lem}
 
 \begin{proof}
 For every point $P\in C$ the restriction ${\mathcal E}_{|\{ P \}\times {\mathbb P}^1}$ is a vector bundle on ${\mathbb P}^1$,
and therefore a direct sum of line bundles. 
Let us write the fiber over the generic $P$ as 
\begin{equation} 
	\label{decP1}
	{\mathcal E}_{|\{ P \}\times {\mathbb P}^1}={\mathcal O}_{{\mathbb P}^1}(\alpha_1)^{r_1}\oplus \cdots \oplus {\mathcal O}_{{\mathbb P}^1}(\alpha_l)^{r_l}, \hskip20pt \alpha _1>\dots >\alpha_l 
\end{equation}
Recall that  the point corresponding to $f$ is unobstructed if and only if $H^{1}(\P^{1},f^{\ast}T_{M})=0$.  By Lemma \ref{sun-lem}, we have that
$$H^{1}( {\mathbb P}^1, f^{\ast}T_{M}) \cong H^{1}(\P^{1},{\cH}om_{0}(\cE)|_{\{P\} \times \P^{1}})$$
 $$\cong H^{1} ({\mathbb P}^1, {\mathcal O}_{{\mathbb P}^1}^{r-1} \oplus \displaystyle\bigoplus_{i \neq j}{\mathcal O}_{{\mathbb P}^1}(\alpha_i-\alpha_j)^{r_ir_j} )
  \cong \displaystyle\bigoplus_{i \neq j}H^{1}({\mathbb P}^1,{\mathcal O}_{{\mathbb P}^1}(\alpha_i-\alpha_j))^{r_ir_j}$$
 Since this vanishes if and only if $|\alpha_i-\alpha_j|\le1$ for all $i$ and $j,$ the result follows.
 \end{proof}
 
 It follows that up to tensoring with the pull back of a line bundle on ${\mathbb P}^1$, we can assume that on an unobstructed component 
  the restriction to the generic fiber is either trivial or ${\mathcal O}_{{\mathbb P}^1}(1)^{r_1}\oplus {\mathcal O}_{{\mathbb P}^1}^{r-r_1}$ for some $r_1 > 0.$  
  
  Corresponding to the decomposition of the generic fiber in equation (\ref{decP1}), we have a relative Harder-Narasimhan filtration for ${\mathcal E}$ with respect to 
  $p : C \times \P^{1} \to C$: 
  \[0={\mathcal E}_0\subset {\mathcal E}_1\subset\dots\subset {\mathcal E}_l={\mathcal E}\]
The successive quotients 
\begin{equation}\label{eqquot} {\mathcal F}_i={\mathcal E}_i/{\mathcal E}_{i-1}\end{equation} 
are each torsion-free with generic splitting type ${\mathcal O}_{{\mathbb P}^1}(\alpha_i)^{r_i}$.
Therefore, ${\mathcal F}'_i :={\mathcal F}_i\otimes p_2^*{\mathcal O}_{{\mathbb P}^1}(-\alpha_i)$ has generically trivial splitting type for each $i.$

A vector bundle on $C\times {\mathbb P}^1$ whose restriction to a general fiber of the projection to $\P^{1}$ is semistable of rank $r$ and degree $d$ gives rise to a rational curve $f : \P^{1} \to M$.
From (2.1), (2.2) in \cite{S},   the degree of the pull back of $\cE$ with respect to the anticanonical bundle can be computed as the discriminant of $\cE$:
$$ \deg(f) = \Delta ( {\mathcal E})=2rc_2( {\mathcal E})-(r-1)c_1( {\mathcal E})^2
  	=2r\sum_{i=1}^nc_2( {\mathcal F}_i')+2\sum_{i=1}^{n-1}(\rank ( {\mathcal E}_i)d-\deg ( {\mathcal E}_i)r)(\alpha _i-\alpha_{i+1})  $$
As $d=h\bar d, r=h\bar r$  with 	$h$ the greatest common divisor of $d,r$,  the degree $k$ of the rational curve as defined in \ref{defdeg} is  
\begin{equation}
	\label{degree}  
	k=\bar r\sum_{i=1}^nc_2( {\mathcal F}_i')+\sum_{i=1}^{n-1}(\rank ( {\mathcal E}_i)\bar d-\deg ( {\mathcal E}_i)\bar r)(\alpha _i-\alpha_{i+1})  
\end{equation}

 \medskip

\begin{lem} 
 	\label{pbfC}
If a torsion free sheaf ${\mathcal E} $ on a ruled surface has generic trivial splitting type, then $c_2( {\mathcal E})\ge 0$
with equality if and only if ${\mathcal E}$ is the pullback of a locally free sheaf on $C$, i.e.~ $\mathcal{E}$ has trivial splitting type on {\bf each} fiber.
\end{lem}

\begin{proof} 
	See the proof of  \cite{GL} Lemma 1.4 (or \cite{S}, Lemma 2.2)  and \cite{H}.
\end{proof}

  Recall that given a vector bundle $E$ on $C$, a point $P\in C$, and a surjective morphism $\phi : E \to \cO_{P},$ the associated elementary transformation $E'$ is the kernel of $\phi.$ 
   In particular $E'$ fits into an exact sequence
  \begin{equation}
  	0 \to E' \to E \to \cO_{P} \to 0
  \end{equation}

\begin{lem}\label{mustbeeltr}
Denote the successive quotients from equation (\ref{eqquot}) by $\cF_{i}$.
 	There exist a finite number of elementary transformations
 	$$0\to  {\mathcal F}_{i1}\to  {\mathcal F}_i\to  {\mathcal O}_{\{ P_{i1}\}\times {\mathbb P}^1}(\beta_1)^{r'_l}\to 0,\
 	0\to  {\mathcal F}_{i2}\to  {\mathcal F}_{i1}\to  {\mathcal O}_{\{ P_{i2}\}\times {\mathbb P}^1}(\beta_2)^{r'_2} \to 0,\dots $$
	$$ \dots , 0\to  {\mathcal F}_{ik_i}\to  {\mathcal F}_{ik_i-1}\to  {\mathcal O}_{\{ P_{ik_i}\}\times {\mathbb P}^1}(\beta_i)^{r'_{k_i}}\to 0$$
 	with $  {\mathcal F}_{ik_i}$ the pull-back of a vector bundle on $C$.
\end{lem}

\begin{proof}   
	Recall that $ {\mathcal F}_i'$ has generic trivial splitting type.  So, from Lemma \ref{pbfC}, $c_2(  {\mathcal F}'_i)\ge 0$
	 with equality if and only if $ {\mathcal F}'_i$ is the pull back of a locally free sheaf on $C$. 

	If  $c_2( {\mathcal F}'_i)> 0$, then the vector bundle is not the pull-back of a vector bundle on $C$.
	Hence, from Lemma   \ref{pbfC},  the fiber over a certain  $P\in C$ has nontrivial splitting type.
  	$$( {\mathcal F}'_i)_{|\{ P \}\times {\mathbb P}^1}={\mathcal O}_{{\mathbb P}^1}(\gamma_1)^{r_1}\oplus {\mathcal O}_{{\mathbb P}^1}(\gamma_2)^{r_2}\oplus\dots
  	\oplus {\mathcal O}_{{\mathbb P}^1}(\gamma_p)^{r_p}, \ \gamma_1<\dots<\gamma_p .$$
  	Consider the natural map $ {\mathcal F}'_i\to  {\mathcal O}_{\{ P\}\times {\mathbb P}^1}(\gamma_1)^{r_1}$ and the corresponding exact sequence
  	$$0\to  {\mathcal F}_{i1}\to  {\mathcal F}'_i\to  {\mathcal O}_{\{ P\}\times {\mathbb P}^1}(\gamma_1)^{r_1}\to 0$$
  		As $ {\mathcal F}'_i$ has trivial splitting type and the degree on each fiber is the same, $\sum_{j=1}^p\gamma_jr_j=0$. 
  	Then, the assumption $\ \gamma_1<\dots<\gamma_p $ implies $\gamma_1 <0$. 
 	From the exact sequence defining $ {\mathcal F}_{i1}$,  
	$$c_2( {\mathcal F}_{i1})=c_2( {\mathcal F}'_i)+r_1\gamma_1< c_2( {\mathcal F}'_i).$$
	Take $\beta_1$ in the statement of the lemma to be $\gamma_1$.

As the cokernel of the injective map $ {\mathcal F}_{i1}\to  {\mathcal F}'_i$ is concentrated on a fiber, the generic splitting type of $ {\mathcal F}_{i1}$ is still trivial. 
If  $c_2( {\mathcal F}_{i1})> 0$, we repeat the process.
We obtain a sequence of bundles ${\mathcal F}_{i1}, {\mathcal F}_{i2}, \dots $ with exact sequences as in the statement of the lemma and
$$\dots < c_2( {\mathcal F}_{i3})<c_2( {\mathcal F}_{i2})<c_2( {\mathcal F}_{i1})< c_2( {\mathcal F}'_i).$$
	From Lemma \ref{pbfC}, $c_2( {\mathcal F}_{ik})\ge 0$ for all $k$ with equality only if $ {\mathcal F}_{ik}$ is the pull back of a sheaf on $C$.
	As $ c_2( {\mathcal F}'_i)$ is finite, the process needs to stop.
	The process can be continued so long as  $c_2( {\mathcal F}_{ik})>0$. 
	Hence, there exists a $k_l$ such that $c_2( {\mathcal F}_{ik_l})=0$ and then $ {\mathcal F}_{ik_l}$ is the pull back of a sheaf on $C$.
	Hence, ${\mathcal F}'_i$ can be obtained by doing elementary transformations from the pull back of a bundle on $C$.
\end{proof}

\begin{cor}
Given a family of rational curves in $M$, one can find families of vector bundles and divisors  $C$ so that 
the rational curves live in spaces of successive extensions and elementary transformations.
\end{cor}

\section{Unobstructed components}\label{secuncomp}

\begin{defn}
If $E$ is a vector bundle of rank $r$ and degree $d$ on $C,$ then for each positive $r'<r$  the \textit{$r'$ Segre invariant} of $E$ is defined as 

\begin{equation*}
	s_{r'}(E) := \min\Biggl\{\begin{vmatrix} r' & r \\ d' & d \end{vmatrix} : {\exists}~ \textnormal{subbundle }E' \subset E\textnormal{ of rank }r'\textnormal{ and degree }d'\Biggr\}
\end{equation*}
\end{defn}

Note that $E$ is stable if and only if $s_{r'}(E) > 0$ for all positive $r' < r.$  If $E$ is a generic stable vector bundle of rank $r$ and degree $d$, we have from Satz 2.2 of \cite{L1} and Th\'{e}or\`{e}me 4.4 of \cite{Hi} that
\begin{equation*}
	r'(r-r')(g-1)\le s_{r'}(E)<r'(r-r')(g-1)+r,  \hskip20pt \ s_{r'}(E) \equiv r'd \ {\rm mod} \ r
\end{equation*}
One has the following results {\cite{RT}}:

\begin{prop} 
\label{Lange} 
When $0 < s \le r'(r-r')(g-1)$ the \textit{Segre locus}
\begin{equation*}
	{\rm S}_{(r',s)}(r,d) := \{ E \in U(r,d) : s_{r'}(E) = s\}
\end{equation*}
is nonempty of codimension $r'(r-r')(g-1)-s$ in $U(r,d).$  Moreover, the generic element $E$ of ${\rm S}_{(r',s)}(r,d)$ is an extension of the form
\begin{equation*}
	0 \to E' \to E \to {E}'' \to 0
\end{equation*}
where $E', {E}''$ are generic elements of $U(r',d')$ and $U(r-r',d-d'),$ respectively (here $d'=\frac{dr'-s}{r}$).  For all such $(r',s)$ we have the inclusion
\begin{equation}
	\label{inclLanloc}
	{\rm S}_{(r',s)}(r,d) \subset \overline{{\rm S}_{(r',s+r)}(r,d)}, \hskip20pt s<r'(r-r')(g-1); 
\end{equation}
\begin{equation*}	
	{\rm S}_{(r',s)}(r,d)=U(r,d), \hskip20pt \ s\ge r'(r-r')(g-1)
\end{equation*}
\end{prop}

\begin{lem} 
\label{lem:stab}
Let $L$ be a fixed line bundle on a curve $C$.  Given generic stable vector bundles $E_1, E_2$ of ranks $r_1, r_2$ and degrees $d_1, d_2$, respectively, with $\frac {d_1}{r_1}<  \frac {d_2}{r_2}$ and $\det E_1\otimes \det E_2=L$
and a generic extension  $0\to E_1\to E\to E_2\to 0$, then  $E$ is stable.
In fact, the loci of non-stable bundles inside the space of extensions has codimension at least $r_1r_2(g-1)$.
\end{lem}
 
\begin{proof}  
By Proposition \ref{Lange}, the  stratification of $M$ by the Segre invariant satisfies 
 \begin{equation*}
 	{\rm S}_{(r',s)}(r,d) \subseteq \overline{{\rm S}_{(r',s+r)}(r,d)}
\end{equation*}
when $s<r'(r-r')(g-1).$  As the moduli space of vector bundles of given rank and degree is nonsingular at any stable point, it suffices to prove the result for the smallest values of $s$,
  namely $0<s\le r$.
 The result is known  without fixing the determinant (p.493 of  \cite{RT}).
 Two spaces of vector bundles with  fixed determinant are isomorphic if the degrees are the same (or simply, congruent modulo $r$).
Therefore, the result is also true with the assumption of fixed determinant.
 \end{proof}
 
 \begin{lem}
 	\label{stabtorext}
  	Assume that $E$ is a generic vector bundle of rank $r$, degree $d+1$ and  Segre invariant $s<r_1(r-r_1)(g-1)$.
  	Then a generic elementary transformation of $E$ is a vector bundle  of rank $r$ and degree $d$ with Segre invariant $s+(r_1-r)$.
  	Moreover, If $E$ is generic with given invariant $s$, its elementary transformation is generic with given invariant $s+r_1-r$.
  \end{lem}
  
  \begin{proof}
 	A generic vector bundle with Segre invariant $s$ corresponds to a generic extension $0\to E_1\to E\to E_2\to 0$ 
 	with $E_1, E_2$ of ranks $r_1, r_2=r-r_1$ and degrees $d_1, d_2=d-d_1$ and  $r_2d_1-r_1d_2=s$.
 	The subbundle $E_1$ is unique with the condition that it has this rank and degree. 
	 A generic elementary transformation does not preserve $E_1$;
 	hence, the generic elementary transformation has invariant $s+(r_1-r)$ (e.g.~ \cite{BL}, Lemma 1.5).
 
Considering  the dual vector space, the process can be reversed, hence a generic element  in  ${\rm S}_{(r_1,s+r_1-r)}(r,d)$ must come from a generic element in  ${\rm S}_{(r_1,s)}(r,d)$.
 \end{proof}

\begin{prop}
\label{prop:fam-ext}
Let $r \geq 2$. Given generic vector bundles $E_1, E_2$ of ranks $r_1, r_2=r-r_1$ and degrees $d_1, d_2=d-d_1$ with $\frac {d_1}{r_1}<  \frac {d_2}{r_2}$ and $\det E_1\otimes \det E_2=L$.
On  $C \times \P^1$ consider a family of extensions of the form
\begin{equation}\label{equnivext}
0 \to p_{1}^{\ast}E_1 \otimes p_{2}^{\ast}\cO_{\P^1}(1) \to \cE \to p_{1}^{\ast}E_2 \to 0
\end{equation}
Consider a point in the space of extensions  over $C \times \P^1$ as a curve in $M$. The degree of this curve is  $\bar dr_1-d_1\bar r.$
\end{prop}
\begin{proof}
The stability of the general extension in this family follows from Lemma \ref{lem:stab}. 
In fact, as the loci of non-stable extensions has codimension at least two, there are whole rational lines contained in the stable locus.
 The degree is computed in equation (\ref{degree}).
\end{proof}

\begin{lem} 
\label{lem:dimh0} 
Consider an exact sequence of vector bundles 
\[  0 \to E_1 \to E \to E_2 \to 0 \]
If $E_1, E_2$ are semistable and $\mu(E_1)< \mu(E_2)$, then $h^0({\cH}om(E_2, E_1))=0.$
\end{lem}

\begin{proof}  
The image of a non-trivial morphism $E_2\to E_1$ would be both a quotient of $E_2$ and a subbundle of $E_1$. 
From semistability $\mu(E_2)\le \mu(Im f)\le \mu (E_1)$ which is incompatible with the assumptions.
\end{proof}

\begin{prop}
\label{dimfamext}
There is a family of maps from rational curves to $M$ as described  in  equation (\ref{equnivext}) with varying $E_1, E_2$. 
The family is parametrized by the Grassmannian of lines of a projective extension space over the space of pairs of bundles of fixed product determinant.
It is unobstructed and has the expected dimension (writing $k$ for the degree  as in (\ref{degree}))
$$\dim M+2hk,\  k=\bar dr_1-d_1\bar r.$$ 
\end{prop}
\begin{proof}

While there are no Poincar\'e bundles on the moduli spaces $U_C(r_i,d_i)$ when $r_i, d_i$ are not coprime, 
from \cite{NR} Prop 2.4, there exists an \'etale cover  $U_1$ of $U_{C}(r_1,d_1)$
 such that there is a universal bundle ${\mathcal E}_1$ on $U_1\times C$. 
  Similarly,  there exists an \'etale cover $U_2$ of $U_{C}(r_2,d_2)$ and a universal bundle ${\mathcal E}_2$ on $U_2\times C$. 
   In what follows we will abuse notation and identify elements of $U_{1}$ and $U_{2}$ with their images under the associated \'{e}tale covers.  We define
$$ {\mathcal U}_{L} :=   \{(E_{1},E_{2}) \in U_1\times U_2 : \det E_1\otimes \det E_2 \cong { L}\}$$
Since this is a fiber of a surjective map from $U_{1} \times U_{2}$ to ${\rm Pic}^{d}(C)$, we have that
$$\dim({\mathcal U}_{L}) = \dim(U_{1})+\dim(U_{2})-g = r_1^2(g-1)+1+r_2^2(g-1)+1-g $$
  We will be using the natural projection maps 
  $$p_1 : {\mathcal U}_{L} \times C \to {\mathcal U}_{L}, p_2 : {\mathcal U}_{L} \times C\to C, \pi_{i} : {\mathcal U}_{L} \to U_{i} (i=1,2)$$  
From  Lemma \ref{lem:dimh0} and Grauert's theorem,  
 \begin{equation*}
 	R^{1}{p_1}_{\ast}{\cH}om((\pi_{2} \times 1_{C})^{\ast}\cE_{2},(\pi_{1} \times 1_{C})^{\ast}\cE_{1})
 \end{equation*}
  is a vector bundle on ${\mathcal U}_{L}$ whose fiber over $(E_{1},E_{2})$ is  ${\rm Ext}^{1}(E_{2},E_{1}).$ 
   Its rank is  
$$dim {\rm Ext }^{1}(E_{2},E_{1}) =  h^1(E_{2}^{\vee} \otimes E_{1})=  r_1r_2(g-1)+r_1d_2-r_2d_1$$  
Consider the projective bundle 
\begin{equation} 
	\label{eqextsp}  
	\pi : \mathbb{P} := {\mathbb P}(R^{1}{p_1}_{\ast}{\cH}om((\pi_{2} \times 1_{C})^{\ast}{\mathcal E}_2,
  (\pi_{1} \times 1_{C})^{\ast}{\mathcal E}_1)) \to \cU_{L}
 \end{equation} 
There is a canonical extension on $\mathbb{P}\times C $
\begin{equation}\label{eqcanext} 
0\to ((\pi_{1} \circ \pi) \times 1_{C})^{\ast} {\mathcal E}_1\otimes {\mathcal O}_{{\mathbb P}}(1)\to \mathcal E\to ((\pi_{2} \circ \pi) \times 1_{C})^{\ast} {\mathcal E}_2\to 0\end{equation}
By Lemma \ref{lem:stab} the general  extension  is stable outside a locus of codimension at least 2, hence there are lines  in $\mathbb{P} $ entirely contained in the stable locus. 
The restriction of ${\mathcal E}$ to the fiber over  a point of ${\mathcal U}_L$ gives a vector bundle on $C $.
 Therefore, for every line in $\mathbb{P} $ and every morphism from $\mathbb{P} ^1$ to this line, we obtain a map from $\mathbb{P} ^1$ to $M$.
 As the line moves in  ${\mathbb P}$ and the morphism from $\mathbb{P} ^1$ to this line moves, we obtain a family of maps from the rational line to  $M$.
 
 The restriction of the canonical extension in equation (\ref{eqcanext}) to  $\mathbb{P} ^1\times C$ where $\mathbb{P} ^1$ is a line in $\mathbb{P} $ 
shows that for a fixed $P\in C$,  the restriction of $ \mathcal E$ to  ${\mathbb P} ^1\times \{ P \}$ is of the form 
 $ {\mathcal O}_{{\mathbb  P}^1}(1)^{r_1}\oplus {\mathcal O}_{{\mathbb P}^1}^{r_2}$.
Hence, from Lemma \ref{unobs}, the component we are constructing is unobstructed.

We compute the dimension of the family  of lines in $\mathbb{P}$ and add to this the dimension of the linear group of $\mathbb{P}^1 $ 
$$\dim~Gr(\mathbb{P}^1,{\mathbb{P}}) =\dim~{\mathcal U}_{L}+\dim~{\mathbb{P}({\rm Ext}^{1}(E_{2},E_{1})})+\dim~{\rm Aut}({\mathbb P}^1)=  $$
$$  r_1^2(g-1)+1+r_2^2(g-1)+1-g+2[r_1r_2(g-1)+r_1d_2-r_2d_1-2]+3$$
 $$=(r^2-1)(g-1)+2[r_1d_2-r_2d_1]=\dim(M)+2hk$$
\end{proof}

 Our next goal is to consider rational curves in $M$ whose general point is an  extension of a torsion sheaf by a vector bundle of rank $r$.
 Using again the correspondence between rational curves in $M$ and vector bundles on 
   $C \times \P^1$, we can consider families of extensions of the form
\begin{equation}\label{extskysc}
0 \to p_{1}^{\ast}E' \otimes p_{2}^{\ast}\cO_{\P^1}(1) \to \cE \to p_{1}^{\ast} {\mathcal O}_D\to 0
\end{equation}
for a divisor $D$ on $C$ of degree $t$.

  \begin{prop} \label{torsext}
  There is a family of maps from rational curves to $M$ as described  in (\ref{extskysc}) with varying $E', D$. 
The family is parametrized by an extension space over the space of pairs of vector bundles and divisors of fixed degree with  fixed product determinant.
The family  is unobstructed and has the expected dimension 
$$\dim M+2hk,\  k=\bar r\deg D.$$ 
 \end{prop}

\begin{proof}
Consider the Hilbert scheme ${\mathcal H}_t$ of divisors of degree $t$ on $C$ with universal subscheme ${\mathcal D}$ on ${\mathcal H}_t\times C$.
Let $U$ be a covering of the moduli space of vector bundles of rank $r$ and degree $d-t$ such that  a Poincar\'e universal bundle ${\mathcal P}$ exists on $U\times C$.
Define
 $$S=\{ (D, E) \in {\mathcal H}_t\times  U|  \ (\wedge^rE)(D)=L\},\ \pi_1: S\to  {\mathcal H}_t, \ \pi_2 :S\to U.$$
On $S\times C\times {\mathbb P}^1$ (with projections $p_i, \ i=1,2,3$), consider the space  of extensions
\begin{equation*}
	{\mathbb P}={\mathbb P}({\rm Ext}^1(( (\pi_1\circ p_{1})\times p_2)^{\ast}{\mathcal O}_{\mathcal D},
 ((\pi_2\circ p_1)\times p_2)^{\ast}{\mathcal P} \otimes p_3^*{\mathcal O}_{{\mathbb P}^1}(1))
\end{equation*} 

The universal exact sequence takes the form

\begin{equation}
\label{eq:torsion-ext}
0 \to  ((\pi_2\circ p_1)\times p_2)^{\ast}{\mathcal P} \otimes p_3^*{\mathcal O}_{{\mathbb P}^1}(1) \to \cE \to 
( (\pi_1\circ p_{1})\times p_2)^{\ast}{\mathcal O}_{\mathcal D} \to 0
\end{equation}
From Lemma \ref{stabtorext}, the generic extension restricted to the fiber over a point in $S$ gives rise to a stable vector bundle on $C$.

From  (\ref{degree}), the generic point in ${\mathbb P}$ parameterizes rational curves in $M$ of degree $\bar rt.$

From the definition of $S$,  
$$\dim S=r^{2}(g-1)+1+t-g=(r^{2}-1)(g-1)+t$$ 
On the other hand, if $D=\sum_{i=1}^tP_i$ then
\begin{equation*} 
	{\rm Ext}^1({\mathcal O}_D, E\otimes {\mathcal O}_{{\mathbb P}^1}(1)) \cong \displaystyle\bigoplus  _{i=1}^t{\rm Ext}^1({\mathcal O}_{P_i}, E\otimes {\mathcal O}_{{\mathbb P}^1}(1))
\end{equation*}
An extension consists of a collection of $t$ extensions. 
Changing any of them by multiplication with a constant does not change the vector bundle.
Therefore the dimension of the family is 
\begin{equation*}
	\dim S+\dim{\rm Ext}^1({\mathcal O}_D, E\otimes {\mathcal O}_{{\mathbb P}^1}(1))-t= (r^{2}-1)(g-1)+t+2tr-t
\end{equation*}
\begin{equation*}
	=(r^{2}-1)(g-1)+2tr  =\dim M+2hk.
\end{equation*}
 \end{proof}


Consider now rational curves of higher degree $a$ inside the space of extensions defined in equation \ref{eqextsp}.
We will see that the images of these  rational curves in $M$ fill the whole of $M$   only when  
the degree  $k$ is divisible by $(g-1)r_1(r-r_1)$ for some $r_1, 1\le r_1\le r-1$ (see  Theorem \ref{excomp}).
In those particular situation of degrees, we obtain the generalization of the ``almost nice  component'' in \cite{C}.

Consider now the combination of the two constructions
$$0 \to p_{1}^{\ast}E_1 \otimes p_{2}^{\ast}\cO_{\P^1}(1) \to {\mathcal E}' \to p_{1}^{\ast}E_2 \to 0$$
$$0 \to {\mathcal E}'  \otimes p_{2}^{\ast}\cO_{\P^1}(1) \to {\mathcal E} \to p_{1}^{\ast} {\mathcal O}_D\to 0$$

\begin{lem}\label{mixed}
There is a family of maps from rational curves to $M$ as described  above with varying $E_1, E_2, D$. 
The family has dimension smaller than expected and therefore it is not a component of the Hilbert scheme of maps from ${\mathbb P}^1$ to $M$.
\end{lem}
\begin{proof} The degree can be computed from equation (\ref{degree}) as 
$$hk=r_1d-rd_1+r\deg D$$
The computation of the dimension of the family is similar to the cases above. It is given as
$$\dim U(r_1, d_1)+\dim U(r_2, d_2)+\deg D-g+2h^1(E_2^*\otimes E_1)+r\deg D-\deg D$$
Using that $d=d_1+d_2+\deg D$ and the expression for $k$, the dimension of the family is given as 
$$\dim M+2hk-2r_1\deg D< \dim M+2hk$$
 \end{proof}

\begin{thm}
\label{teornumcomp}
Given $d, r>0, \ k>0$ integers, let $h$ be the greatest common divisor of $r, d$.
There exist $h$ different families of unobstructed maps from rational curves to $M$ of degree $k$. 
\end{thm}
\begin{proof}
Write $d=h\bar d, \ r=h\bar r$. As $\bar d, \bar r$ are relatively prime, there exist unique integers $r_0$ and $d_0$ such that $0 \leq r_0 < \bar{r}$ and $(r_{0},d_{0})$ is a solution of  
\begin{equation}
	\label{eq:dioph}
	\bar{d}x-\bar{r}y = k
\end{equation}
There are exactly $h$ choices for a solution $(r_{1},d_{1})$ of (\ref{eq:dioph}) satisfying $0 \leq r_{1} < r.$

Assume first that $r_1>0$. Defining $d_2 := d-d_1, r_2 := r-r_1$, we have 
\begin{equation}
	d_2r_1-r_{2}d_{1}=dr_1-rd_{1}=hk > 0
\end{equation}
 It follows at once that $\frac{d_1}{r_1} < \frac{d_2}{r_2}$; in particular, the conditions of Proposition  \ref{prop:fam-ext} are satisfied.  Moreover, from Proposition \ref{dimfamext} there is a family of maps corresponding to points in the space of extensions of the pull back
 of a vector bundle of rank $r_2$ and degree $d_2$ by the pull back of a  vector bundle of rank $r_1$ and degree $d_1$ tensored with ${\mathcal O}_{{\mathbb P}^1}(1)$.
 We know that this family has the right dimension and the generic point is unobstructed. 
 
If $r_1=0$, $k$ is divisible by $\bar r$. We can then use Proposition \ref{torsext} instead of Proposition  \ref{prop:fam-ext} for the construction of one of the components.

We need to check that these are the only unobstructed components. 
From Lemma \ref{unobs}, an unobstructed component has one or two steps in the Harder-Narasimhan filtration.
From Lemma \ref{mustbeeltr}, the pieces we use to build extensions come from elementary transformations of pull backs of vector bundles on the curve $C$.
From Lemma \ref{mixed}, there is no need to do elementary transformation in the two step extensions.
Therefore, there are no other unobstructed components.
 \end{proof}


\section{Additional components}\label{secadcomp}
From equations (\ref{decP1}), (\ref {degree}) and Lemma \ref{unobs}, any additional components would come from families in which the restriction of the vector bundle to the generic ${\mathbb P}^1$ 
is direct sum of line bundles of at least three different degrees or line bundles of two degrees that differ in more than one unit.
We look at the second case first.
\medskip
 \begin{lem}
\label{prop:exttw}
Let $r \geq 2$ and assume that $E_1, E_2$ are generic semistable vector bundles of respective ranks $r_1, r_2$ and degrees $d_1, d_2$ with $r_{1}+r_{2}=r,$ $d_{1}+d_{2}=d,$ $\frac {d_1}{r_1}<  \frac {d_2}{r_2}$ and $\det E_1\otimes \det E_2 \cong L$.
On  $C \times \P^1$ consider an extension of the form
	\begin{equation}
	\label{eq:exttw}
		0 \to p_{1}^{\ast}E_1 \otimes p_{2}^{\ast}\cO_{\P^1}(a) \to \cE \to p_{1}^{\ast}E_2 \to 0
	\end{equation}
	
	\begin{itemize}
		\item[(i)]{For general $y \in \P^{1}$ the restriction $\cE|_{C \times y}$ is a stable bundle, and the degree of the associated rational curve in $M$ is $k=a[\bar dr_1-d_1\bar r].$}
		
		\item[(ii)]{If $a > 1,$ the family of rational curves on $M$ obtained from {\rm (\ref{eq:exttw})} by varying $E_1$ and $E_2$ is obstructed and has dimension 
		$$  \dim{M}+hk+(a-1)r_1r_2(g-1)+[r_1d_2-r_2d_1].$$}
	\end{itemize}
\end{lem}

\begin{proof} 
	The stability of the general extension is proved as in Proposition \ref{prop:fam-ext}, while the degree computation follows from equation (\ref{degree}); this proves (i). 
	 We now turn to (ii), whose proof is as in Proposition \ref{dimfamext}.
 	
	Using again ${\mathcal U}_{L}$ for the space of pairs of bundles with determinant of the product the given $L$,
	consider the projective bundle over ${\cU}_{L}$
	$$\mathbb{P}_a:= {\mathbb P}(R^{1}{\pi}_{\ast}{\cH}om((\pi_{2} \circ p_1\times p_2)^{\ast}{\mathcal E}_2,
  	(\pi_{1} \circ p_1\times p_2)^{\ast}{\mathcal E}_1\otimes   p_3^*{\mathcal O}_{{\mathbb P}^1}(a))).$$
	There is a canonical extension on $\mathbb{P}_a\times C\times {\mathbb P}^1 $
	\[ 0\to ((\pi_{2} \circ p_1)\times p_2)^{\ast} {\mathcal E}_1\otimes  p_3^*{\mathcal O}_{{\mathbb P}^1}(a)\to 
	{\mathcal E}\to( (\pi_{2} \circ p_1)\times p_2)^{\ast} {\mathcal E}_2\to 0\]
	The restriction of ${\mathcal E}$ to the fiber over  a point of ${\mathcal U}_L$ gives a vector bundle on $C\times {\mathbb P}^1 $ and therefore,
	a rational curve in $M$.  As the point moves in  ${\mathcal U}_L$, we obtain a family of rational curves in $M$.

	The restriction of ${\mathcal E}$  to the fiber over  a point of ${\mathcal U}_{L}\times C$  is of the form
 	${\mathcal O}_{{\mathbb P}^1}(a)^{r_1}\oplus {\mathcal O}_{{\mathbb P}^1}^{r_2}$.
	Hence, from Lemma \ref{unobs}, if $a>1$, the family we are constructing is obstructed.

	By Lemma \ref{lem:stab} the general  extension  is stable. 
	From the correspondence between vector bundles on $C\times {\mathbb P}^1$ and rational curves in $M$, $\mathbb{P}_a$ gives rise to a family of such maps. 
	$$  \dim (\mathbb{P}_a )=\dim(U_{\cL})+\dim ( \mathbb{P}({\rm Ext}^{1}(E_{2},E_{1}(a))))=  $$
	$$  r_1^2(g-1)+1+r_2^2(g-1)+1-g+(a+1)[r_1r_2(g-1)+r_1d_2-r_2d_1]-1=$$
	$$ =(r^2-1)(g-1)+(a-1)r_1r_2(g-1)+ah(r_1\bar d-\bar r d_1)  +[r_1d_2-r_2d_1]=$$
	$$=\dim(M)+hk+(a-1)r_1r_2(g-1)+[r_1d_2-r_2d_1]$$
\end{proof}  

We now consider families of extensions of the form
\begin{equation}\label{extskysctw}
0 \to p_{1}^{\ast}E' \otimes p_{2}^{\ast}\cO_{\P^1}(a) \to \cE \to p_{1}^{\ast} {\mathcal O}_D\to 0
\end{equation}
for a divisor $D$ on $C$ of degree $t$.

  \begin{lem} \label{torsexttw}
  There is a family of maps from rational curves to $M$ as described  in (\ref{extskysctw}) with varying $E', D$. 
The family is parametrized by an extension space over the space of pairs of vector bundles and divisors of fixed degree with  fixed product determinant
and has  dimension 
$$\dim M+hk+r\deg D,\  k=a\bar r\deg D$$ 
For $a>1$, the family  is obstructed and is not a component of the space of maps from ${\mathbb P}^1$ to $M$.
 \end{lem}
\begin{proof}
The proof is as in Proposition \ref{torsext}, replacing ${\mathbb P}$ by 
$${\mathbb P}_a={\mathbb P}(Ext^1(( (\pi_1\circ p_{1})\times p_2)^{\ast}({\mathcal O}_{\mathcal D}),
 ((\pi_2\circ p_1)\times p_2)^{\ast}( {\mathcal P})\otimes p_3^*({\mathcal O}_{{\mathbb P}^1})(a)))$$  
Lemma \ref{stabtorext}, can still be applied to prove the stability of the generic extension.

From  (\ref{degree}), the generic point in ${\mathbb P}_a$ parameterizes rational curves in $M$ of degree $a\bar rt.$

The dimension of the family is 
$$\dim S+Ext^1({\mathcal O}_D, E\otimes {\mathcal O}_{{\mathbb P}^1}(a))-t=
(r^{2}-1)(g-1)+t+(a+1)tr-t
 =\dim M+hk+tr$$
 For this family to be a component of the space of rational maps to $M$, it would have to have dimension at least the expected dimension.
 This would imply that $tr\ge atr$ which in turn implies $a\le 1$.
 \end{proof}

\begin{thm} 
\label{excomp} 
For $a>1$, if the family described in Proposition \ref{prop:exttw} is an (obstructed) component of the space of maps from ${\mathbb P}^1$ to $M$ of degree $k=a(r_1d-rd_1)$,
 then a vector bundle in the image rational curve in $M$ is not generic except when $k$ is divisible by $(g-1)r_1(r-r_1)$ for some $r_1, 1\le r_1\le r-1$.
 \end{thm}
\begin{proof}
We first show that the dimension we found in Proposition \ref{prop:exttw} is larger than the expected dimension $\dim M +2hk$ if and only if $r_1d-rd_1\le r_1(r-r_1)(g-1)$.
As the expected dimension is the smallest dimension a component of the space of maps can have, this will suffice to prove that the family is not a component of the space of maps 
if $r_1d-rd_1> r_1(r-r_1)(g-1)$.
We will only need to deal with the case when that inequality is an equality.
The condition on the dimension can be written as 
$$\dim(M)+hk+(a-1)r_1r_2(g-1)+[r_1d_2-r_2d_1] \ge  \dim(M)+2hk$$
Using that $hk=a(r_1 d- r d_1)$, this is equivalent to 
$$(a-1)r_1(r-r_1)(g-1)+(r_1 d - rd_1)\ge a (r_1 d - rd_1) $$
As we are assuming $a>1$, the inequality is preserved by dividing by $a-1$. We obtain the equivalent equation 
	\begin{equation}
		\label{eq:almost-nice}
 		r_1(r-r_1)(g-1)\ge r_1d-rd_1
	\end{equation}
When the inequality is strict, this implies (see equation  (\ref{inclLanloc})) that the corresponding vector bundle is special. 
When the inequality is an equality, we obtain the special case in which $k$ is divisible by $(g-1)r_1(r-r_1)$.
\end{proof}

\begin{rem}
	When $r=2$ and $d=1,$ equality in (\ref{eq:almost-nice}) implies that $g$ is even, so that we recover the full statement of Theorem 1.6 in \cite{C}.
\end{rem}

\begin{cor} 
\label{excompr=2} 
If $r=2, a>1$,  the family described in Proposition \ref{prop:exttw} is an (obstructed) component of the space of maps from ${\mathbb P}^1$ to $M$ of degree $k=a(d-2d_1)$
 if and only if $d-2d_1<g-1$, or equivalently, when the vector bundle in the image rational curve in $M$ is not generic.
 \end{cor}
\begin{proof}
The only-if part has already been proved. 
It remains to show that, under the given conditions, the family is actually a component of the space of rational curves. 
For this, it suffices to show that it is not in the closure of a larger family of such curves.

Assume first that $r$ is general and that $r_1d-rd_1<r_1(r-r_1)(g-1)$ as in (\ref{excomp}).
By construction of the family in (\ref{eq:exttw}), a point in the image rational curve in $M$ is a
vector bundle on $C$ which is an extension of a vector bundle of rank $r_2$ and degree $d_2$   by a vector bundle of rank $r_1$ and degree $d_1$.
From the inclusion in equation (\ref{inclLanloc}), if this family is contained in a larger component, the $d_1$ should decrease.
On the other hand, a family of vector bundles on ${\mathbb P}^1$ of the form ${\mathcal O}_{{\mathbb P}^1}(a)^{r_1}\oplus {\mathcal O}_{{\mathbb P}^1}^{r_2}$
can be deformed by moving some of the degree of  ${\mathcal O}_{{\mathbb P}^1}(a)$ to some of the  ${\mathcal O}_{{\mathbb P}^1}$.
As we can normalize by tensoring with a line bundle on ${\mathbb P}^1$, this has the effect of decreasing the $a$.
If we assume that the $k$  stays constant, then $ad_1$ is constant and therefore, $d_1$ can be written in terms of $a$.
 Writing the dimension in Proposition \ref{prop:exttw} as a function of $a$ alone, we notice that it is an increasing function of $a$.
 Hence, the family corresponding to a value of $a$ cannot be in the closure of the family corresponding to a different value.
 While for arbitrary $r$ one could deform the family to a family in which the decomposition of the bundle to the rational curve has more summands, this cannot happen for rank two. 
 This concludes the proof in this case.
\end{proof}

 Consider now vector bundles on  $C\times {\mathbb P}^1$ whose restriction to the generic ${\mathbb P}^1$ is direct sum of line bundles of at least three different degrees. 
 Up to tensoring with a line bundle on ${\mathbb P}^1$, we can assume that one of the summands is trivial, 
 $${\mathcal O}_{{\mathbb P}^1}(a_1+\dots +a_{l-1})^{r_1}\oplus {\mathcal O}_{{\mathbb P}^1}(a_2+\dots +a_{l-1})^{r_2}\oplus \dots  
 \oplus {\mathcal O}_{{\mathbb P}^1}(a_{l-1})^{r_{l-1}}\oplus {\mathcal O}_{{\mathbb P}^1}^{r_l}$$

We would be considering extensions of the form 
 
\begin{equation}
	\label{ext3}
	\begin{matrix} 
		0 \to p_{1}^{\ast}E_1 \otimes p_{2}^{\ast}\cO_{\P^1}(a_1) \to{ \mathcal E }_2' \to p_{1}^{\ast}E_2 \to 0, & 
		0 \to {\mathcal E}_2 ' \otimes p_{2}^{\ast}\cO_{\P^1}(a_2) \to {\mathcal E}_3 '   \to p_{1}^{\ast}E_3 \to 0 \\
		\dots \ \ \ ,& \   0 \to {\mathcal E}_{l-1} ' \otimes p_{2}^{\ast}\cO_{\P^1}(a_{l-1}) \to {\mathcal E}  \to p_{1}^{\ast}E_l \to 0
	\end{matrix}
\end{equation}
  
We can obtain a family of such extensions by considering successive spaces of extensions similarly to the construction in \ref{dimfamext}:
  
     \begin{lem}
\label{fam2se}
Given positive integers $a_1,\dots, a_{l-1}, r_1,\dots, r_l$, arbitrary integers $d_1,\dots, d_l$ with 
$$r_1+\dots+r_l=r, \ d_1+\dots+d_l=d, \  \frac{d_1}{r_1}<\frac{d_2}{r_2}<\dots <\frac{d_l}{r_l}$$
there is a family of maps from rational curves to $M$ whose generic restriction to ${\mathbb P}^1$
is $${\mathcal O}_{{\mathbb P}^1}(a_1+\dots +a_{l-1})^{r_1}\oplus{\mathcal O}_{{\mathbb P}^1}(a_2+\dots +a_{l-1})^{r_2}\oplus \dots
\oplus {\mathcal O}_{{\mathbb P}^1}(a_{l-1})^{r_{l-1}}\oplus {\mathcal O}_{{\mathbb P}^1}^{r_l}$$
The degree is obtained from  
  $$hk=\sum_{i<j}(r_id_j-r_jd_i)(a_i+a_{i+1}+\dots+a_{j-1})$$
The family is obstructed if $l\ge 3$ and has dimension 
$$\dim M+\sum_{i<j}(r_id_j-r_jd_i)(a_i+a_{i+1}+\dots +a_{j-1}+1)+[\sum_{i<j}r_ir_j(a_i+a_{i+1}+\dots +a_{j-1}-1)](g-1)$$ 
\end{lem}

\begin{proof}   
 Denote by  $U(r_i,d_i)$ a suitable cover of the moduli space of vector bundles of rank $r_i$ and degree $d_i$ such that on $C\times U(r_i,d_i)$ a Poincare bundle ${\mathcal E}_i$ exists.
 Consider the space ${\mathcal U}_{L}$ of $l$-ples of bundles with determinant of the product the given $L$. 
 Denote by $\pi_1, \pi_2,\dots,  \pi_l$ the projection of $U(r_1,d_1)\times  U(r_2,d_2)\times \dots \times U(r_l,d_l)$ onto $U(r_1,d_1), U(r_2,d_2), \dots ,U(r_l,d_l)$ respectively
 as well as the restriction of these projections to ${\mathcal U}_{L}$.
Denote by $p_1, p_2, p_3$ the projection of $ C \times {\mathbb P}^1\times {\mathcal U}_{L}$ onto  $ C ,  {\mathbb P}^1, {\mathcal U}_{L}$ respectively.
  Consider the projective bundle over $ C \times {\mathbb P}^1\times {\mathcal U}_{L}$
$$\mathbb{P}_1:= {\mathbb P}(R^{1}{p}_{3\ast}{\cH}om( p_1\times (\pi_{2} \circ p_3))^{\ast}{\mathcal E}_2,
  ( p_1\times (\pi_{1} \circ p_3))^{\ast}{\mathcal E}_1\otimes (  p_2)^*({\mathcal O}_{{\mathbb P}^1}(a_1))).$$
with canonical extension on $ C\times {\mathbb P}^1\times \mathbb{P}_1 $
\[ 0\to ( p_1\times (\pi_{1} \circ p_3))^{\ast} {\mathcal E}_1\otimes   p_2^*{\mathcal O}_{{\mathbb P}^1}(a_1)\to {\mathcal E}_1'\to
 ( p_1\times (\pi_{2} \circ p_3))^{\ast} {\mathcal E}_2\to 0.\] 
 
 We then construct a bundle $\mathbb{P}_{2}$ over $\mathbb{P}_{1}$ by considering extensions of the pull back of ${\mathcal E}_3$ by 
  ${\mathcal E}_1'\boxtimes  {\mathcal O}_{{\mathbb P}^1}(a_2)$. 
  More generally, we construct   $\mathbb{P}_{j}$ as a projective bundle over $  {\mathbb P}_{j-1}, j=2,\dots, l-1$ (we omit pull back maps and write $\boxtimes$ instead):
$$\mathbb{P}_j:= {\mathbb P}(R^{1}{p}_{3\ast}{\cH}om( {\mathcal E}_{j+1},
  {\mathcal E}_j'\boxtimes {\mathcal O}_{{\mathbb P}^1}(a_j))).$$
with canonical extension on $ C\times {\mathbb P}^1\times \mathbb{P}_j $
\begin{equation}\label{canext} 0\to  {\mathcal E}'_{j-1}\boxtimes   {\mathcal O}_{{\mathbb P}^1}(a_j)\to {\mathcal E}_j'\to
  {\mathcal E}_{j+1}\to 0.\end{equation} 
  We will write ${\mathcal E}={\mathcal E}'_{l-1}$
 
From the correspondence between vector bundles on $C\times {\mathbb P}^1$ and rational curves in $M$, $\mathbb{P}_{l-1}$ gives rise to a family of  maps from 
${\mathbb P}^1$ to $M$. 

Our next goal is to compute the dimension of the family we constructed. Note that 
$$\dim(U_{\cL})=\sum_{i=1}^l( r_i^2(g-1)+1)-g=[r^2-1-2\sum_{i\not= j}r_ir_j](g-1)+(l-1)=\dim M-2(g-1)\sum_{i\not= j}r_ir_j+l-1$$
$$  \dim \mathbb{P}_{l-1} =\dim(U_{\cL})+\dim{(\text fib}{\mathbb P}_1\to (U_{\cL}))+\dim{(\text fib}{\mathbb P}_2\to {\mathbb P}_1)+\dots+\
dim{(\text fib}{\mathbb P}_{l-1}\to {\mathbb P}_{l-2}) $$

 The fibers of the projection ${\mathbb P}_1\to {\mathcal U}_{L}$ are 
 $${\mathbb P}(Ext^1_{C\times {\mathbb P}^1}(E_2, E_1\boxtimes {\mathcal O}_{{\mathbb P}^1}(a_1)))=
 {\mathbb P}(H^1(C, E_2^*\otimes E_1)\times H^0({\mathbb P}^1, {\mathcal O}_{{\mathbb P}^1}(a_1))).$$
 As $\mu(E_1)<\mu(E_2)$,  $h^0(C, E_2^*\otimes E_1)=0$.
  Hence, the dimension of the fibers of the projection ${\mathbb P}_1\to {\mathcal U}_{L}$ is
  $$[r_1d_2-r_2d_1+r_1r_2(g-1)](a_1+1)-1$$
  Similarly, the fibers of the projection ${\mathbb P}_2\to {\mathbb P}_1$ are
   ${\mathbb P}(Ext^1_{C\times {\mathbb P}^1}(E_3, {\mathcal E}'_2\boxtimes {\mathcal O}_{{\mathbb P}^1}(a_2)))$.
   In order to compute the dimension of these fibers, we need to use the tautological sequence defining ${\mathcal E}'_2$
   tensored with the pull back of the dual of $E_3$ and $ {\mathcal O}_{{\mathbb P}^1}(a_2)$.
 We omit pull back maps and write $\boxtimes$ instead:
\[ 0\to   E_1\otimes E_3^*\boxtimes  {\mathcal O}_{{\mathbb P}^1}(a_1+a_2)\to 
\mathcal E'_2\boxtimes  E_3^*\boxtimes  {\mathcal O}_{{\mathbb P}^1}(a_2)
\to  E_2\otimes   E_3^*\boxtimes  {\mathcal O}_{{\mathbb P}^1}(a_2)\to 0\]  
   We obtain that 
$$\dim Ext^1_{C\times {\mathbb P}^1}(E_3, {\mathcal E}'_2\boxtimes {\mathcal O}_{{\mathbb P}^1}(a_2))=
h^1(C, E_3^*\otimes E_1)h^0({\mathbb P}^1, {\mathcal O}_{{\mathbb P}^1}(a_1+a_2))+
h^1(C, E_3^*\otimes E_2)h^0({\mathbb P}^1, {\mathcal O}_{{\mathbb P}^1}(a_2))=$$
  $$=[r_1d_3-r_3d_1+r_1r_2(g-1)](a_1+a_2+1)+[r_2d_3-r_3d_2+r_2r_3(g-1)](a_1+1)-1$$
The dimension of the remaining fibers would be computed similarly.

Therefore,  the dimension of the family is 
$$\dim M-2(g-1)\sum_{i\not= j}r_ir_j+l-1+[r_1d_2-r_2d_1+r_1r_2(g-1)](a_1+1)-1+$$
$$+[r_1d_3-r_3d_1+r_1r_3(g-1)](a_1+a_2+1)+[r_2d_3-r_3d_2+r_2r_3(g-1)](a_1+1)-1]$$
$$+\dots +[r_1d_l-r_ld_1+r_1r_l(g-1)](a_1+\dots +a_{l-1}+1)+\dots +[r_{l-1}d_l-r_ld_{l-1}+r_{l-1}r_l(g-1)](a_{l-1}+1)-1 =$$
$$=\dim M+\sum_{i<j}(r_id_j-r_jd_i)(a_i+a_{i+1}+\dots +a_{j-1}+1)+[\sum_{i<j}r_ir_j(a_i+a_{i+1}+\dots +a_{j-1}-1)](g-1)$$

We can compute the degree using  equation (\ref{degree}).
 Multiplying by $h$ and using that 
 $$d=d_1+d_2+\dots +d_l, r=r_1+r_2+\dots +r_l $$
 we obtain 
  $$hk=(r_1 d-d_1r)a_1+((r_1+r_2)d-(d_1+d_2)r)a_2+\dots + ((r_1+\dots +r_{l-1})d-(d_1+\dots+d_{l-1})r)a_{l-1}=$$
  $$=(r_1 d_2-r_2d_1)a_1+(r_1 d_3-r_3d_1)(a_1+a_2)+\dots +(r_1 d_l-r_ld_1)(a_1+\dots +a_{l-1})+$$
$$+(r_2d_3-r_3d_2 )a_2+\dots +(r_2 d_l-r_ld_2)(a_2+\dots +a_{l-1})  \dots  +(r_{l-1} d_l-r_ld_{l-1})a_{l-1}=$$
$$=\sum_{i<j}(r_id_j-r_jd_i)(a_i+a_{i+1}+\dots+a_{j-1})$$

From Lemma \ref{unobs}, the family we are constructing is obstructed if $l\ge 3$ or $a_i\ge 2$
\end{proof}

\begin{thm} 
	\label{comp2se} 
		If the family described in Lemma \ref{fam2se} is an (obstructed) component of the space of maps from ${\mathbb P}^1$ to $M$, 
 		then a vector bundle in the image rational curve in $M$ is not generic.
\end{thm}
 
 \begin{proof}
For a family as in \ref{fam2se} to be a component of the Hilbert scheme of maps of ${\mathbb P}^1$ to $M$, 
its dimension needs to be at least as large as the expected dimension $\dim M+2hk$.
 This condition is 
 $$\dim M+\sum_{i<j}(r_id_j-r_jd_i)(a_i+a_{i+1}+\dots +a_{j-1}+1)+[\sum_{i<j}r_ir_j(a_i+a_{i+1}+\dots +a_{j-1}-1)](g-1)\ge$$ 
$$\dim M+2\sum_{i<j}(r_id_j-r_jd_i)(a_i+a_{i+1}+\dots+a_{j-1})$$
 This can be rewritten as 
  $$\sum_{i<j}(r_id_j-r_jd_i-r_ir_j(g-1))(a_i+a_{i+1}+\dots +a_{j-1}-1)\le 0$$
  Using that 
  $$a_i+a_{i+1}+\dots +a_{j-1}-1=(a_i-1)+(a_{i+1}-1)+\dots +(a_{j-1}-1)+(j-i-1)$$
  and taking common factor the $a_i-1$, we obtain
 $$(a_1-1)\sum_{i=2}^l (r_1d_i-r_id_1-r_ir_1(g-1))+(a_2-1)[\sum_{i=3}^l [r_1d_i-r_id_1-r_ir_1(g-1)+r_2d_i-r_id_2-r_ir_2(g-1)]+\dots$$
 $$\dots+(a_j-1))[\sum_{i=j+1}^l\sum _{k=1}^j (r_kd_i-r_id_k-r_ir_k(g-1))+\dots +(a_{l-1}-1)\sum _{k=1}^l(r_kd_l-r_ld_k-r_lr_k(g-1))+$$
  $$+\sum_{i<j-1}(j-i-1)(r_id_j-r_jd_i-r_ir_j(g-1))\le 0$$
  Regrouping the terms, this gives rise to 
 \begin{equation}
 	\label{longineq}
		(a_1-1)[r_1\sum_{i=2}^l d_i-\sum_{i=2}^lr_id_1-r_1\sum_{i=2}^lr_i(g-1)] 
\end{equation}
\begin{equation*}
		+ (a_2-1)[(r_1+r_2)\sum_{i=3}^ld_i-\sum_{i=3}^lr_i(d_1+d_2)-(r_1+r_2)\sum_{i=3}^lr_i(g-1)]+\dots
\end{equation*}
 $$\dots+(a_j-1)[ (\sum _{i=1}^j r_i) (\sum_{k=j+1}^ld_k)- (\sum_{k=j+1}^lr_k)( \sum _{i=1}^j d_i)- (\sum_{k=j+1}^lr_k)  (\sum _{i=1}^j r_i) (g-1))+\dots $$
$$ \dots +(a_{l-1}-1)(\sum _{k=1}^{l-1}r_k)d_l-r_l(\sum _{k=1}^{l-1}d_k)-r_l(\sum _{k=1}^{l-1}r_k)(g-1))$$
 $$+\sum_{i<j-1}(j-i-1)(r_id_j-r_jd_i-r_ir_j(g-1))\le 0$$
 \bigskip
 
 For every point $(x, y)\in  {\mathbb P}^1\times {\mathbb P}_{l-1}$, we obtain a vector bundle on $C$ by considering the restriction 
 $E={\mathcal E}_{|C\times \{ x\}\times \{ y\}}$.
 We want to show that under the above conditions, $E$ is not generic in $M$ for generic $(x, y)\in  {\mathbb P}^1\times {\mathbb P}_{l-1}$.
 
 From  exact sequence (\ref{canext}) for $j=l-1$, $E$ is an extension of $E_l$ by $E'_{l-1}={\mathcal E}'_{l-1||C\times \{ x\}\times \{ y\}}$.
 The latter is a vector bundle of rank $r_1+\dots+r_{l-1}$ and degree $d_1+\dots +d_{l-1}$.
 From Propositon \ref{Lange}, if $E$ is not special , we have 
 
 \begin{equation*}
 	\label{inl-1} 
		\sum _{k=1}^{l-1}r_kd_l-r_l(\sum _{k=1}^{l-1}d_k)-r_l(\sum _{k=1}^{l-1}r_k)(g-1)\ge 0
\end{equation*}

More generally,  write  $E'_{j}={\mathcal E}'_{j||C\times \{ x\}\times \{ y\}}$. 
 Then  $E'_{j}$ is a vector bundle of rank $r_1+\dots+r_{j}$ and degree $d_1+\dots +d_{j}$.  
 Assembling together the injective maps from (\ref{canext}) for $j, j+1, \dots l-1$, 
 we obtain an inclusion $E'_{j}\to E$ whose cokernel has rank $r_{j+1}+\dots+r_{l}$ and degree $d_{j+1}+\dots +d_{l}$.
 From Proposition \ref{Lange}, if $E$ is not special, we have 

 \begin{equation}
 	\label{inj}  
		(\sum _{i=1}^j r_i) (\sum_{k=j+1}^ld_k)- (\sum_{k=j+1}^lr_k)( \sum _{i=1}^j d_i)- (\sum_{k=j+1}^lr_k)  (\sum _{i=1}^j r_i) (g-1))\ge 0. 
\end{equation}

\begin{claim} 
	Multiplying equation (\ref{inj}) by $r_1+\dots +r_{j-1}+r_{j+2}+\dots +r_l$ and adding for $j=1,\dots, l-1$, we obtain
  	\begin{equation}
		\label{insum} 
			\sum_{i<j-1}(j-i-1)[r_id_j-r_jd_i-r_ir_j(g-1)]\ge \frac{g-1}r\sum_{1\le m<n<p\le l}r_mr_nr_p>0   
	\end{equation}
  \end{claim}
  
 Note then that  inequalities  (\ref{inj}) and (\ref{insum})  are incompatible with (\ref{longineq}). This will complete the proof of the Theorem.
 
 \begin{proof} (of the claim) For $i<j$, write $A_{ik}=r_id_k-r_kd_i-r_ir_k(g-1)$.
 Note that 
 \begin{equation}\label{IdentAs} \text{If  } l<m<n, \  r_nA_{lm}-r_lA_{mn}=r_mA_{ln}-r_lr_mr_n(g-1)\end{equation}
 Multiplying equation (\ref{inj}) by $r_1+\dots +r_{j-1}+r_{j+2}+\dots +r_l$ and adding for $j=1,\dots, l-1$, we obtain
 \[  \sum_{j=1}^{l-1} (r_1+\dots +r_{j-1}+r_{j+2}+\dots +r_l)[ \sum _{i=1}^j \sum_{k=j+1}^lA_{ik}]\ge 0. \]
 This can be written as 
\begin{equation}\label{suin}\sum_{1\le i<k\le l} [(k-i)(r_1+\dots +r_{i-1})+(k-i-1)r_i +(k-i-2)(r_{i+1}+ \cdots + r_{k-1})
\end{equation}
\begin{equation*}
 + \cdots +(k-i-1)r_{k}+(k-i)(r_{k+1}+\dots +r_{l})]A_{ik} \ge 0\end{equation*}
 Note that 
\[ \text{If  } m<n<p, \  \ \ r_pA_{mn}+r_mA_{np}=r_nA_{mp}-r_mr_nr_p(g-1).\]
Therefore, for $t>k$, $r_tA_{ik}$ can be combined with $r_iA_{kt}$ to give rise to $r_kA_{it}-r_ir_kr_t(g-1)$.
 When  the process is carried out for all $i,k, i<k$ once for each $t>k$, , one of the terms $r_sA_{ik}$ will be used up for each $s<i$
  and new terms $r_vA_{ik}$ will be gained for $i<v<k$.
  Then, inequality (\ref{suin}) becomes
 
 \[  \sum_{1\le i<k\le l} (k-i-1)rA_{ik}-(g-1)\sum_{1\le m<n<p\le l}r_mr_nr_p  \ge  0. \]
 as claimed.
 \end{proof}
  \end{proof}

\section{Rational curves through two generic points}\label{secratc2points}

A question of interest in the study of rational curves on Fano varieties is the minimum degree of a rational curve through two generic points. 
In \cite{KMM}. Kollar, Miyaoka and Mori showed that the degree is bounded by a quadratic expression on the dimension. 
For $M$, this bound can be greatly improved and is in fact linear on the dimension.

\begin{prop} 
	\label{prop:rat-conn}
	Given two generic points of $M$, there is a rational curve  containing the two points of degree $(\frac{r^2}2-1)(g-1)$ if $r$ is even and degree $\frac{3r^2-3}2(g-1)$ if $r$ is odd.
\end{prop}
 \begin{proof} Given $E_1, E_2\in M$ generic, we want to find $E', E''$ such that we have exact sequences
 $$0\to E'\to E_1\to E''\to 0,\ 0\to E'\to E_2\to E''\to 0$$
 If these extensions exist, then  there is a line in the projective space of extensions containing the two given ones and its image is a rational curve in $M$ containing both $E_1, E_2$.

The dimension of the space of extensions for a fixed $E', E'' $ is $r'd''-r''d'+r'r''(g-1)$. 
Moreover, the dimension of the fibers of the map of the space of extensions to $M$ are well behaved, that is as small dimensional as possible (\cite{RT}).
Therefore, we need only $r'd''-r''d'+r'r''(g-1)\ge (r^2-1)(g-1)$. 
Equivalently, 
$$hk=r'd-rd'\ge (r^2-1-r'(r-r'))(g-1).$$
The smallest value of the right hand side is obtained for $r'$ as close as possible to $\frac r2$ which gives the statement in the proposition.
\end{proof}



\end{document}